\documentclass[11pt,reqno]{amsart}
\usepackage{amssymb,amsmath,mathtools}
\usepackage{xcolor}
\theoremstyle{plain}
\newtheorem{theorem}{Theorem}[section]
\newtheorem{corollary}[theorem]{Corollary}
\theoremstyle{definition}
\newtheorem{assumption}[theorem]{Assumption}
\newtheorem{definition}[theorem]{Definition}
\newtheorem{remark}[theorem]{Remark}
\numberwithin{equation}{section}
\def \au {\rm}
\def \ti {\it}
\def \jou {\rm}
\def \bk {\it}
\def \no#1#2#3 {{\bf #1} (#3), #2.}
\def \eds#1#2#3 {#1, #2, #3.}
\title[Weak Solutions of Linear Differential Equations in Hilbert Spaces]
{An Observation About Weak Solutions of\\
Linear Differential Equations in Hilbert Spaces}
\author[V. Pata and J.T. Webster]
{Vittorino Pata and Justin T. Webster}
\address{Politecnico di Milano, Milano, Italy}
\email{vittorino.pata@polimi.it}
\address{University of Maryland, Baltimore County, Baltimore, MD}
\email{websterj@umbc.edu}
\thanks{V.P.\ has been partially supported by the Italian MIUR-PRIN
Grant 2020F3NCPX “Mathematics for industry 4.0 (Math4I4)”}

\begin{document}

\begin{abstract}
In this note we propose a definition of weak solution for an abstract Cauchy problem in a Hilbert space,
and we discuss existence and uniqueness results.
\end{abstract}

\maketitle

\section{Existence of Weak Solutions}

\noindent
Let $H$ be a separable real Hilbert space with inner-product $(\cdot,\cdot)_H$ and norm $\|\cdot\|_H$,
respectively,
and let $A: \mathcal D(A) \subset H \to H$ be a densely defined
closed linear operator. The space $H$ is thought to be embedded in a bigger
Hilbert space, allowing us to continuously extend $A$ outside of its domain, as is typically the case
for differential operators.
Assume that we are given the abstract evolution in $H$ on the time-interval $[0,T]$
\begin{equation}
\label{cauchy}
\begin{cases}
\dot x = Ax +f, \\
x(0)=x_0\in H,
\end{cases}
\end{equation}
in the unknown $x=x(t)$, where $f=f(t)$ is a forcing term.
The aim of this note is to provide a convincing definition of weak solution which fully captures the structure of the Cauchy problem \eqref{cauchy}, and then
present a corresponding existence and uniqueness result. We make a baseline assumption.

\begin{assumption}
\label{Axx}
Let $W\hookrightarrow H$ be another Hilbert space, possibly coinciding with $H$.
\begin{itemize}
\item[(i)] $A$ extends to a bounded operator from $W$ to its image $AW$. Moreover,
identifying $H$ with its dual space $H^*$, there exists a Hilbert space $V$ such that $V^*=AW$,
and the following chain of (continuous and dense) embeddings hold:
$$V \hookrightarrow W\hookrightarrow H =H^*\hookrightarrow W^*\hookrightarrow V^*.$$
\item[(ii)] There exist $a\geq b>0$ such that, for all $x\in\mathcal D(A)$,
\begin{equation}
\label{AxxIN}
-(Ax,x)_H \geq b\|x\|_W^2-a\|x\|_H^2.
\end{equation}
\end{itemize}
\end{assumption}

In the sequel, $\langle\cdot,\cdot\rangle$ will stand for the duality pairing between $V^*$ and $V$,
extending the inner-product in $H$, whenever both defined.
We introduce the bilinear form $a(\cdot, \cdot):W\times V\to\mathbb R$
$$a(x,\tilde x) =-\langle Ax, \tilde x\rangle,$$
whose continuity follows from  (i).
Concerning the external force, we assume that
$$f\in L^2(0,T;W^*).$$

\begin{definition}[$W$-Weak Solution]
\label{defWEAK}
A function $x \in L^{\infty}(0,T;H) \cap L^2(0,T;W)$ is a $W$-weak solution of~\eqref{cauchy}
if and only if $x(0)=x_0$ and, for a.e.\ $t \in [0,T]$ and every test $y\in V$,
$$
\langle \dot x(t),y\rangle+a(x(t),y)=\langle f(t),y\rangle.
$$ From our assumptions,
such an equality dictates that
$\dot x\in L^2(0,T;V^*)$, and the embedding $H^1(0,T;V^*)\hookrightarrow C([0,T];V^*)$
ensures that $x(0)$ has meaning.
\end{definition}

In general, the space $W$ complying with Assumption~\ref{Axx}
is not unique. This will be clear in Example~\ref{heq}, where there are infinitely many choices of $W$.
For any viable $W$, we produce a legitimate notion of a (weak) solution.
Among them, we can locate the best possible one, corresponding
to the minimum with respect to the relation of inclusion of such spaces $W$.
This minimum provides the notion of solution that fully captures the features of the equation, as
it exploits the maximal information which can be extracted weakly from the operator $A$.
In principle, such a minimum might
not exist (if it is only an infimum), but it seems very unlikely to find
a concrete example where this sorry situation occurs.
We agree to simply call \emph{weak solution} the $W$-weak solution corresponding to the minimum $W$
or, equivalently, the strongest among $W$-weak solutions.

\begin{theorem}
\label{teoEXI}
Under Assumption~\ref{Axx}, problem~\eqref{cauchy} admits at least one $W$-weak solution.
\end{theorem}

\begin{proof}
The proof is based on a standard Galerkin approximation scheme.
We select a basis $\{h_n\}$ of the space $H$ made by elements in $V$, and we denote
by $\Pi_n$ the self-adjoint projection from $H$
onto $F_n$, where $F_n$ is the $n$-dimensional subspace
of $H$ generated by the vectors $h_1,\ldots,h_n$. We agree to call the pair of spaces and projections
$(F_n,\Pi_n)$ a \emph{Galerkin family}.
Usually, the basis $\{h_n\}$ is chosen
to be orthonormal, although this is not strictly necessary.
Appealing to the classical theory of ODEs, for every $n\in {\mathbb N}$
there is a unique function $x_n\in AC([0,T];F_n)$ with prescribed
initial condition $x_n(0)=\Pi_n x_0$ that solves
the equation
\begin{equation}
\label{Gal}
\langle \dot x_n(t),y\rangle+a(x_n(t),y)=\langle f(t),y\rangle,
\end{equation}
for a.e.\ $t\in[0,T]$ and every test $y\in F_n$. Observe that the duality pairings appearing above are actually just $H$ inner-products at the level of $F_n$.
Such an $x_n$ is sometimes called a \emph{Galerkin approximate solution}.
Choosing as test function $y=x_n$, from point (ii) of Assumption~\ref{Axx}, together with Young's inequality, we obtain
the differential inequality
$$
\frac{d}{dt}\|x_n\|^2_H=-2a(x_n,x_n)+2\langle f,x_n\rangle
\leq -b\|x_n\|_W^2+2a\|x_n\|_H^2+\frac{1}{b}\|f\|_{W^*}^2.
$$
As $\|x_n(0)\|_H\leq \|x_0\|_H$, an application of the standard Gronwall lemma yields an \emph{energy estimate}
\begin{equation}
\label{eee}
\|x_n\|_{L^{\infty}(0,T;H)}+\|x_n\|_{L^2(0,T;W)}\leq C\big[\|x_0\|_H +\| f\|_{L^2(0,T;W^*)}\big],
\end{equation}
for some $C>0$ independent of $n$. Then, from~\eqref{Gal} and the embedding
$W^*\hookrightarrow V^*$, we also deduce the uniform bound
$$\|x_n\|_{H^1(0,T;V^*)}\leq C\big[\|x_0\|_H +\| f\|_{L^2(0,T;W^*)}\big],
$$
up to redefining $C$.
Accordingly, there exists
$x \in L^{\infty}(0,T;H) \cap H^1(0,T;V^*) \cap L^2(0,T;W)$ such that, up to a subsequence,
we have the weak$^*$ and weak convergences
$$
x_n {\stackrel{_*}{\rightharpoonup}} x \in L^{\infty}(0,T;H),\qquad
x_n \rightharpoonup x \in H^1(0,T;V^*),\qquad
x_n \rightharpoonup x\in L^2(0,T;W).
$$
At this point, we fix $m\in\mathbb N$ and we take $y\in F_m$.
Passing to the limit in~\eqref{Gal},
we obtain
$$\langle \dot x(t),y\rangle+a(x(t),y)=\langle f(t),y\rangle.$$
Since this is true for every $m$, we conclude that the latter equality holds for all test $y\in V$.
Finally, in light of the embedding $H^1(0,T;V^*)\hookrightarrow C([0,T];V^*)$,
we establish the equality $x(0)=x_0$. Thus, $x$ is a $W$-weak solution.
\end{proof}

Any solution obtained through this limiting procedure is called a \emph{Galerkin solution}.
On account of the weak lower semicontinuity of the norm, we infer from~\eqref{eee} an immediate corollary.

\begin{corollary}
\label{corEE}
There exists $C>0$ such that any Galerkin solution $x$ satisfies the energy estimate
\begin{equation}
\label{eeee}
\|x\|_{L^{\infty}(0,T;H)}+\|x\|_{L^2(0,T;W)}\leq C\big[\|x(0)\|_H +\| f\|_{L^2(0,T;W^*)}\big].
\end{equation}
Moreover, $x \in C([0,T];V^*)$.
\end{corollary}

The uniqueness problem is more subtle, since one would hope to test a solution $x$
of the homogeneous problem (with a null initial datum) with $x$ itself. This is possible only if $x$ resides in the
correct test space, $V$, which may not be the case. However, we will see that by
introducing an additional straightforward hypothesis on $a(\cdot,\cdot)$, we
automatically gain uniqueness and boosted  regularity
of the weak solution. Before stating this hypothesis,
and the consequent result, we consider some motivational examples to establish the validity and applicability of our scheme.

\section{Three Examples}
\label{SecEx}

\noindent
Let $X^0$ be a separable real Hilbert space with inner-product $\langle\cdot, \cdot\rangle_0$ and norm
$\|\cdot\|_0$, and let $B$ be a strictly positive self-adjoint
linear operator on $X^0$ with domain $\mathcal D(B)$, compactly embedded
into $X^0$ (hence the spectrum of $B$ comprises only eigenvalues).
For $r\in\mathbb R$, we introduce the compactly nested scale of Hilbert spaces
$$X^r=\mathcal D(B^{\frac{r}2}),\qquad ( u,v)_r=( B^\frac{r}2 u,B^\frac{r}2 v)_0,
\qquad \|u\|_r=\|B^\frac{r}2 u\|_0.$$
If $r>0$, it is understood that $X^{r}$ is the completion of the domain, so that
$X^{-r}$ is the dual space of $X^r$.
The duality pairing between
$X^{-r}$ and $X^r$ will be still denoted by
$\langle\cdot,\cdot\rangle$.
The paradigmatic example we have in mind for such a $B$ is the Dirichlet Laplacian $-\Delta$ acting
on the Hilbert space $L^2(\Omega)$
with $\mathcal D(-\Delta) =H^2(\Omega)\cap H_0^1(\Omega)$, where $\Omega \subset \mathbb R^n$ is a bounded
domain with smooth boundary.

\subsection{The Abstract Heat Equation}
\label{heq}
Given $f\in L^2(0,T;X^{-1})$,
we consider the Cauchy problem in $H=X^0$
$$
\begin{cases}
\dot u = -B u +f, \\
u(0)=u_0\in H,
\end{cases}
$$
which is in the form \eqref{cauchy} with $A=-B$.
For every $u\in X^2$ we have that
\begin{equation}
\label{bilinearHeat}
(B u,u)_{0}=\|u\|^2_1.
\end{equation}
Accordingly, we take
$W=V=X^1$ and $V^*=X^{-1}$, and we define the bilinear form
on $X^1\times X^1$
$$
a(u,\tilde u)=(u, \tilde u)_{1}.
$$
A weak solution is then a function $u\in L^\infty(0,T;X^0)\cap L^2(0,T;X^1)$
such that $u(0)=u_0$ and
$$
\langle \dot u(t),\phi\rangle+(u(t), \phi)_{1}=\langle f(t),\phi\rangle,
$$
for a.e.\ $t\in[0,T]$ and every test $\phi\in X^1$. From Theorem~\ref{teoEXI}, such a solution exists.
Moreover, considering a solution $u$ of~\eqref{cauchy} with $u_0=0$ and $f=0$, we are in the ideal case where we can utilize $u$ itself as a test function,
since $u(t)\in X^1$ for a.e.\ $t\in[0,T]$. This provides the energy estimate~\eqref{eee} with $C=0$,
establishing the uniqueness of the solution as well.

\begin{remark}
It is evident from~\eqref{bilinearHeat} that Assumption~\ref{Axx} would be in place with any $W=X^{r}$ with $r\in[0,1]$.
Hence, although we obtained the weak solution (i.e., the strongest in this framework) by choosing $W=X^1$, we could have
given the notion of $X^r$-solution for any $r\in[0,1)$ as well. The weakest among those is the one corresponding to
$W=X^0$. Namely, for $f\in L^2(0,T;X^0)$,
a $X^0$-weak solution is a function $u\in L^\infty(0,T;X^0)$
such that $u(0)=u_0$ and
$$
\langle \dot u(t),\phi\rangle-(u(t), B\phi)_{0}=(f(t),\phi)_0,
$$
for a.e.\ $t\in[0,T]$ and every test $\phi\in X^2$. (See also the second bullet in Section \ref{finrem}.)
\end{remark}

\subsection{The Abstract Wave Equation}
\label{weq}
Given $g\in L^2(0,T;X^0)$,
we consider the Cauchy problem
$$
\begin{cases}
\ddot u = -B u+g, \\
u(0)=u_0,~~
\dot u(0)=u_1,
\end{cases}
$$
with $u_0\in X^1$ and $u_1\in X^0$.
Setting $x=[u,v]^\top$ and $f=[0,g]^\top$, the above can given the form~\eqref{cauchy}
in $H=X^1\times X^0$, by setting
$$A= \begin{bmatrix} 0 & I \\ -B & 0 \end{bmatrix}
\qquad\text{with domain}\qquad
\mathcal D(A)=X^2\times X^1.
$$
For $x\in X^2\times X^1$, we have
$$-(A x,x)_{X^1\times X^0}=0.
$$
In light of~\eqref{AxxIN}, this forces the equality $W=H=X^1\times X^0$.
Since $AW=X^0\times X^{-1}$,
and since the pivot space of the chain of embeddings is $H$, we obtain
$V=X^2\times X^1$ and $V^*=X^0\times X^{-1}$,
while, for $x=[u,v]^\top$ and $\tilde x=[\tilde u,\tilde v]^\top$, the bilinear form is given by
$$
a(x,\tilde x)=-(v, B\tilde u)_{0}+(u, \tilde v)_{1}.
$$
Then $x=[u,v]^\top$ with $u\in L^{\infty}(0,T;X^1)$
and $v\in L^{\infty}(0,T;X^0)$
is a weak solution if
$x(0)=[u_0,u_1]^\top$ and for every test $y=[\phi,\psi]^\top\in X^2\times X^1$
the equality
$$
(\dot u(t),B\phi)_{0}+\langle \dot v(t),\psi\rangle
-(v, B \phi)_{0}+(u(t), \psi)_{1}=(g(t),\psi)_{0}
$$
holds for a.e.\ $t\in[0,T]$.
Since the hypotheses on $B$ ensure that the vectors of the form $B\phi$ with $\phi\in X^2$
span $X^0$, we can equivalently say that $x$ is a weak solution if
for every test
$y=[\xi,\psi]^\top\in X^0\times X^1$
$$
(\dot u(t),\xi)_{0}+\langle \dot v(t),\psi\rangle
-(v, \xi)_{0}+(u(t),\psi)_{1}=(g(t),\psi)_{0}.
$$ From this definition,
we can easily recover the usual weak form of the wave equation.
Indeed, choosing $\psi=0$ above, we readily obtain the equality $v=\dot u$ in $X^0$. From there,
we choose $\xi =0$ to obtain what might be considered the weak formulation of the wave equation, that is,
$$
\langle \ddot u(t),\psi\rangle
+(u(t), \psi)_{1}=(g(t),\psi)_{0}.
$$
Although the existence follows from Theorem~\ref{teoEXI},
the subtle point comes in obtaining the energy estimate~\eqref{eee} for
any weak solution. In this case, one
should choose  $\psi=\dot u$
as test function in the latter equality, which is {\em forbidden}, since $\psi$ is required to live in $X^1$, whereas
$\dot u$ is in $X^0$ only.
Thus, in order to obtain uniqueness of solutions,
as well as a continuous dependence estimate, some additional (and nontrivial) work is needed; see
what is done, for instance, in~\cite[pp.406--408]{EVA} and \cite[II.4.1, pp.76--79]{TEM}.
Instead, by the method presented below, we will see that
when one has existence of a weak solution via a ``good" Galerkin construction,
as it happens here,
then the weak solution is automatically unique,
satisfies the energy estimate, continuously dependent on the initial data, and has additional regularity.

\subsection{The Abstract Damped Wave Equation}
\label{dweq}
Let $\alpha\in[0,1]$ be fixed. Given a forcing term $g\in L^2(0,T;X^{-\alpha})$,
we consider the Cauchy problem
$$
\begin{cases}
\ddot u = -B u-B^\alpha \dot u+g,\\
u(0)=u_0,~~
\dot u(0)=u_1,
\end{cases}
$$
with $u_0\in X^1$ and $u_1\in X^0$. Of particular interest are the two limit cases $\alpha=0$,
corresponding to the weakly damped wave equation or telegrapher's equation, and
$\alpha=1$ corresponding to the strongly damped wave equation, also known as
Kelvin-Voigt equation.
The problem can be given the form~\eqref{cauchy}
with $H=X^1\times X^0$, by setting
$$A= \begin{bmatrix} 0 & I \\ -B & -B^\alpha  \end{bmatrix}
\qquad\text{with domain}\qquad
\mathcal D(A)=
\left\{ [u,v]^\top \left|
\begin{array}{c}
v\in X^1\\
u + B^{\alpha-1}v \in X^2
\end{array}\right.
\right\},
$$
which becomes $\mathcal D(A)=X^2\times X^1$ whenever
$\alpha\leq 1/2$.
For $x=[u,v]^\top\in \mathcal D(A)$, we now have
$$-(A x,x)_{X^1\times X^0}=\|v\|_{\alpha}^2.
$$
Accordingly, $W=X^1\times X^\alpha$ is the best possible space complying with~\eqref{AxxIN}.
When $\alpha=0$ the only possibility is $W=H$.
Since $AW=X^\alpha\times X^{-1}$,
and since the pivot space is $H$, we obtain
$$V=X^{2-\alpha}\times X^1,\qquad W=X^1\times X^\alpha,\qquad
W^*=X^1\times X^{-\alpha},\qquad V^*=X^\alpha\times X^{-1}.
$$
For $x=[u,v]^\top$ and $\tilde x=[\tilde u,\tilde v]^\top$, the bilinear form is given by
$$
a(x,\tilde x)=-(v, \tilde u)_{1}+(u, \tilde v)_{1}
+(v, \tilde v)_{\alpha}.
$$
A function $x=[u,v]^\top$ with $u\in L^{\infty}(0,T;X^1)$
and $v\in L^{\infty}(0,T;X^0)\cap L^2(0,T;X^\alpha)$
is a weak solution if
$x(0)=[u_0,u_1]^\top$ and for every test $y=[\phi,\psi]^\top\in X^{2-\alpha}\times X^1$
the equality
$$
(\dot u(t),B^{1-\alpha}\phi)_{\alpha}+\langle \dot v(t),\psi\rangle
-(v(t), B^{1-\alpha}\phi)_{1}+(u(t), \psi)_{1}
+(v(t), \psi)_{\alpha}\\
=\langle g(t),\psi\rangle
$$
holds for a.e.\ $t\in[0,T]$. Since $B^{1-\alpha}$ maps $X^{2-\alpha}$ onto $X^\alpha$,
this is the same as saying that $x$ is a weak solution if
for every test $y=[\xi,\psi]^\top\in X^\alpha\times X^1$
$$
(\dot u(t),\xi)_{\alpha}+\langle \dot v(t),\psi\rangle
-(v(t), \xi)_{1}+(u(t), \psi)_{1}
+(v(t), \psi)_{\alpha}\\
=\langle g(t),\psi\rangle.
$$
Similarly to the previous example, we obtain the equality $v=\dot u$ in $X^\alpha$,
along with the usual weak formulation of the damped wave equation, that is,
$$
\langle \ddot u(t),\psi\rangle
+(u(t), \psi)_{1}+(\dot u(t),\psi)_{\alpha}=\langle g(t),\psi\rangle,
$$
Again, existence follows from Theorem~\ref{teoEXI}.
When $\alpha=1$ uniqueness is easily obtained as in Section \ref{heq},
as the solution lives in the test space.

\section{Uniqueness and Regularity}

\noindent
Whenever Assumption~\ref{Axx} is in force, Theorem~\ref{teoEXI} ensures the existence
of a $W$-weak solution to the Cauchy problem in~\eqref{cauchy} via a standard Galerkin construction. The next assumption
will automatically guarantee uniqueness of the solution, among several other improved properties.

\begin{assumption}
\label{good}
In the terminology of the proof of Theorem~\ref{teoEXI}, let there exists a Galerkin family
$(F_n,\Pi_n)$ with the following property:
for all vectors $x\in W$ and all $\tilde x\in V$,
$$ a(x,\Pi_n\tilde x)=a(\Pi_n x,\tilde x),\quad\forall n\in\mathbb N.$$
\end{assumption}

\begin{theorem}
\label{th:main}
Under Assumptions~\ref{Axx} and~\ref{good}, the Cauchy problem~\eqref{cauchy} admits a unique
$W$-weak solution $x$, for any admissible $W$ space. Moreover,
$x \in C([0,T];H)$ and
the map $x_0\mapsto x(t)$ belongs to $C(H;H)$ for any fixed $t\in[0,T]$.
\end{theorem}

\begin{proof}
Due to linearity, uniqueness follows once we prove that any solution to the
homogeneous problem with null initial datum
is  trivial.
Let then $x$ be any such solution (namely, of \eqref{cauchy} with $x_0=0$ and $f = 0$), and set $x_n(t)=\Pi_n x(t)$.
In light of Assumption~\ref{good}, taking a test $y\in F_n$, we have the equality
$$
\langle \dot x_n(t),y\rangle+a(x_n(t),y)=0.
$$
This tells that $x_n$ is a Galerkin approximate solution, which is known to converge (up to a subsequence)
to some Galerkin solution $\hat x$. But since we already have that $x_n(t)\to x(t)$ in $H$, this forces the equality $\hat x=x$.
Hence, $x$ is a Galerkin solution of the homogeneous problem with initial datum $x_0=0$, and
applying Corollary~\ref{corEE} allows us to conclude that $x$ is the null solution. By the same token, once uniqueness is established
(and so all solutions are Galerkin solutions)
we establish the continuous dependence estimate. Indeed, let $x,\hat x$ be two solutions of~\eqref{cauchy}
with initial data $x(0)=x_0$ and $\hat x(0)=\hat x_0$, respectively. Then, their difference $x-\hat x$
is the solution of the homogeneous problem with initial datum $x_0-\hat x_0$, and Corollary~\ref{corEE}
provides the estimate
$$\|x-\hat x\|_{L^\infty(0,T;H)}+\|x-\hat x\|_{L^2(0,T;W)} \leq C\|x_0-\hat x_0\|_H.
$$
We are left to prove the continuity in time.
To this end, let us assume first that $f \in L^2(0,T;H)$. For an arbitrarily given $x_0\in H$,
we construct the Galerkin approximate
solutions $x_n$,
based on the Galerkin family $(F_n,\Pi_n)$.
Now let $n\geq m$. Exploiting
Assumption~\ref{good},
for any test $y\in F_n$, we have the equality
$$\langle \dot x_m(t),y\rangle+a(x_m(t),y)
=\langle \dot x_m(t),\Pi_m y\rangle+a(x_m(t),\Pi_m y)=\langle f(t),\Pi_m y\rangle=\langle\Pi_m f(t),y\rangle.$$
Hence, for any test $y\in F_n$,  the difference $x_n-x_m$ fulfills
$$\langle \dot x_n(t)-\dot x_m(t),y\rangle+a(x_n(t)-x_m(t),y)=\langle f(t)-\Pi_m f(t),y\rangle.
$$
Choosing $y=x_n-x_m$,
and arguing as in the proof of Theorem~\ref{teoEXI},
we obtain the estimate
$$
\sup_{t\in[0,T]}\|x_n(t)-x_m(t)\|_H \leq C\Big[\|\Pi_n x_0-\Pi_m x_0\|_H+\| f- \Pi_mf\|_{L^2(0,T;H)}\Big].
$$
Observe that the last term of the right-hand side above goes to zero as $m\to\infty$, by the Lebesgue dominated convergence theorem.
Therefore, $x_n$ is a Cauchy sequence in $C([0,T];H)$, and so converges to an element $x$ in that space. Such an $x$
is exactly the previously established (unique) solution, taken with initial datum $x_0$.
Now, to deal with the general case, let us  consider a sequence $f_n\in L^2(0,T;H)$ converging to
$f \in L^2(0,T;W^*)$ in the latter norm. This time, we call $x_n$ the solution with initial datum $x_0$ corresponding
to the forcing term $f_n$. A further application of Corollary~\ref{corEE} yields the estimate
$$
\sup_{t\in[0,T]}\|x_n(t)-x_m(t)\|_H \leq C\|f_n- f_m\|_{L^2(0,T;W^*)}.
$$
Again, we conclude that $x_n$ converges to $x\in C([0,T];H)$. And it is standard matter to verify that $x$ is the solution
with initial datum $x_0$ corresponding to the forcing term $f$.
\end{proof}

\begin{remark} It is clear from the above proof that,
if one wants to prove only uniqueness and continuous dependence,
then Assumption~\ref{good} is not needed in its full strength;  it suffices to take only the weaker assumption
$$a(x,\Pi_n\tilde x)=a(\Pi_n x,\Pi_n\tilde x),\quad\forall n\in\mathbb N.$$
\end{remark}

Returning to the examples of Section~\ref{SecEx},
let $\lambda_1\leq\cdots \leq\lambda_n\to\infty$ be the sequence of eigenvalues of $B$,
counted with multiplicity,
and let $\{b_n\}$ be the corresponding sequence of  eigenvectors,
which form (if normalized) a complete orthonormal family in $X^0$ belonging to $X^r$ for all $r>0$.
For the Cauchy problem~\eqref{heq}, the basis $\{h_n\}$ of $H=X^0$ needed to construct $F_n$
is simply $\{b_n\}$.
For the Cauchy problems~\eqref{weq} and~\eqref{dweq}, the orthonormal basis of $H=X^1\times H^0$
is made by the  vectors $h_{2n+1}=b_n/\sqrt{\lambda_n}\,\oplus 0$ and $h_{2n}=0\oplus b_n$.
In all  cases presented here,  our Assumption~\ref{good} is then immediately verified.

\section{Final Remarks}\label{finrem}

\noindent
$\bullet$ Under Assumptions~\ref{Axx} and~\ref{good}, we have a unique
$W$-weak solution of the Cauchy problem~\eqref{cauchy}.
Since all the $W$-weak solutions arise from the same Galerkin scheme, they must coincide.
Hence, the point of locating the best possible $W$ is actually the one of giving a notion
of solution apt to describe all the regularity properties satisfied by the solution itself.

\smallskip
\noindent
$\bullet$ Although we stated our results under the hypothesis $f\in L^2(0,T;W^*)$,
everything works if we take a more general forcing term $f\in L^1(0,T;H)+L^2(0,T;W^*)$.
The only additional ingredient needed in the proofs is a modified form of the Gronwall lemma (see, e.g., \cite{PPV}),
in order to obtain the energy estimate~\eqref{eee}.

\smallskip
\noindent
$\bullet$
In a standard way, the existence of a unique $W$-weak solution implies that the operator $A$ is the infinitesimal generator
of a $C_0$-semigroup $S(t)$ acting on $H$. In which case, it is well-known that if
$f \in L^1(0, T; H)$ then the Cauchy problem~\eqref{cauchy} admits a unique \emph{mild solution},
given by the variation of parameters (Duhamel) formula (see, e.g., \cite{ENG}),
$$x(t)=S(t) x+\int_0^t S(t-s) f(s) d s.$$
Such a mild solution is nothing but the $H$-weak solution, hence, the weakest
possible arising in the framework presented here.

\smallskip
\noindent
$\bullet$ A different definition of weak solution for \eqref{cauchy} has been proposed in~\cite{BAL}.
Namely,
a function $x\in C([0, T]; H)$ is a weak solution of \eqref{cauchy} if and only if $x(0)=x_0$ and for every
$y \in \mathcal D(A^*)$ the map $t \mapsto \langle u(t), y\rangle$ is absolutely continuous on $[0, T]$ and
\begin{equation}\label{b}
\frac{d}{d t}\langle u(t), y\rangle=(u(t), A^*y)_H+(f(t), y)_H,
\end{equation}
for a.e.\ $t \in [0,T]$, where $A^*$ is the adjoint of the operator $A$.
Actually, in~\cite{BAL} it is shown that $A$ is the generator of a $C_0$-semigroup if and only if for every initial datum $x_0\in H$
there exists a unique weak solution in the  sense of \eqref{b}. This definition, although very elegant, is of less
practical use since, in general, finding the adjoint of $A$ is a formidable task.

\smallskip
\noindent
$\bullet$ As already mentioned, one way to establish the existence and uniqueness of weak (mild) solutions
is to prove that the operator $A$ is the infinitesimal generator
of a $C_0$-semigroup $S(t)$. This requires us to apply the Feller-Miyadera-Phillips theorem.
Unfortunately, as clearly stated in~\cite[Comment 3.9, pp.77--78]{ENG}, checking the hypotheses of this theorem
is out of question in the general case. It can be reasonably accomplished only when the resulting
$S(t)$ is an $\omega$-contraction
(in which case the Feller-Miyadera-Phillips theorem boils down to the Hille-Yosida theorem). This means that, in
the operator norm, $S(t)$ fulfills the estimate $\|S(t)\|\leq e^{\omega t}$,
for some $\omega\in{\mathbb R}$. If so, such a semigroup generator $A$ satisfies the estimate
$$(Ax,x)_H \leq \omega \|x\|_H^2,\quad\forall x\in\mathcal D(A).$$
This is (in fact, a particular case of) point (ii) of our Assumption~\ref{Axx}.



\end{document}